\documentclass[11pt,a4paper]{amsart}

\usepackage{times}
\usepackage{hyperref}
\hypersetup{
  colorlinks   = true, %Colours links instead of ugly boxes
  urlcolor     = blue, %Colour for external hyperlinks
  linkcolor    = blue, %Colour of internal links
  citecolor   = blue %Colour of citations
}
\usepackage{amsmath}
\usepackage{amsfonts}
\usepackage{amssymb}
\usepackage{amsthm}
\newtheorem{thm}{Theorem}[section]

\newtheorem{thmx}{Theorem}

\newenvironment{thmp}{\stepcounter{thm}\begin{thmx} }{\end{thmx}}

\newtheorem{cor}[thm]{Corollary}
\newtheorem{lem}[thm]{Lemma}
\newtheorem{prop}[thm]{Proposition}
\newtheorem{teo}[thm]{Theorem}

\newtheorem{Obs}[thm]{Observation}

\theoremstyle{definition}
\newtheorem{defi}[thm]{Definition}
\newtheorem{example}[thm]{Example}
\newtheorem{rem}[thm]{Remark}

\newtheorem{case}{Case}[section]

\DeclareMathOperator{\proj}{proj}
\DeclareMathOperator{\Iso}{Iso}
\DeclareMathOperator{\supp}{supp}
\DeclareMathOperator{\Geo}{Geo}

\DeclareMathOperator{\Opt}{Opt}

\DeclareMathOperator{\Diam}{Diam}

\DeclareMathOperator{\Adm}{Adm}

\DeclareMathOperator{\argmax}{arg\,max}
\DeclareMathOperator{\argmin}{arg\,min}

\DeclareMathOperator{\Prob}{\mathbb{P}}
\DeclareMathOperator{\W}{\mathbb{W}}
\DeclareMathOperator{\Tub}{Tub}
\DeclareMathOperator{\GTB}{GTB}
\DeclareMathOperator{\sIP}{sIP}
\DeclareMathOperator{\Mid}{Mid}
\DeclareMathOperator{\Cut}{Cut}

\usepackage{enumitem}

\usepackage[usenames]{color}

\begin{document}
\title{Isometric rigidity of compact Wasserstein spaces. }
\author{Jaime Santos-Rodr\'iguez}
\address{Max-Planck-Institut f\"ur Mathematik, Bonn, Germany}
\email{jaime@mpim-bonn.mpg.de}
\thanks{The author was supported by research grants MTM2014-57769-3-P, MTM2017-85934-C3-2-P (MINECO), ICMAT Severo Ochoa Project SEV-2015-0554 (MINECO) and a Postdoctoral Fellowship from the MPI Bonn.  }
\date{\today}

% MATH SUBJECT CLASSIFICATION AND KEYWORDS

\subjclass[2010]{53C23, 53C21}
\keywords{Wasserstein distance,  isometry group}
\begin{abstract}
Let $(X,d,\mathfrak{m})$ be a metric measure space. The study of the Wasserstein
space $(\Prob_p(X),\W_p)$ associated to $X$ has proved useful in describing several geometrical properties of
$X.$ In this paper we focus on the study
of isometries of $\Prob_p(X)$ for $p \in (1,\infty)$ under the assumption  that
there is some characterization of optimal maps between measures, the so called Good transport behaviour $\GTB_p$.
Our first result states that the set of Dirac deltas is  invariant under isometries of the Wasserstein space.
Additionally we obtain that the isometry groups of the base Riemannian manifold $M$ coincides with the one of the Wasserstein
space $\Prob_p(M)$ under assumptions on the manifold; namely, for $p=2$ that the sectional curvature is strictly positive
and for general $p\in (1,\infty)$  that $M$ is a Compact Rank One Symmetric Space.
\end{abstract}

\maketitle

%%%%%%%%%%%%%%%%%%%%%%%%%%%%%%%%%%%%%%%%%%%%%%%%%%%%%%%%%%%%%%%%%%%%%%%%%%%%%%%%
%     Introduction and statement of results.
%%%%%%%%%%%%%%%%%%%%%%%%%%%%%%%%%%%%%%%%%%%%%%%%%%%%%%%%%%%%%%%%%%%%%%%%%%%%%%%%
\section{Introduction and statement of results.}\label{Sec.Intro}

The space of probability measures equipped with the Wasserstein metric reflects
several geometrical properties of a metric measure space $(X,d,\mathfrak{m})$ such as;
compactness, existence of geodesics, and non-negative sectional curvature. (see for example \cite{AmbGig}, \cite{Vill}).

A natural question therefore is asking whether it is possible for the Wasserstein space
to be more symmetric than the base space. Consider the following,  if $g: X \rightarrow X$ is an isometry then
it is easy to check that $g_{\#}: \Prob_p(X)\rightarrow \Prob_p(X)$ is also an isometry for any $p \in (1,\infty).$
Therefore ${\#}\Iso(M) \subset \Iso(\Prob_p(X)).$ So more concretely the question  is to determine whether these two groups of
isometries are the same, in such case we will say that $X$ is isometrically rigid.

First of  all, notice that for any map $g: X \rightarrow X,$ $g_{\#}\delta_x= \delta_{g(x)}$ for all $x\in X,$ i.e. the pushforward of a Dirac delta is again a Dirac delta. Hence our first approach should be to determine whether the set of  Dirac deltas is invariant under isometries. Our result in this regard is:

\newtheorem*{thm:A}{Theorem \ref{teo:A}}
\begin{thm:A}
 Let $(X,d,\mathfrak{m})$ be a compact metric measure space with $\GTB_p$ for some $p \in (1,\infty).$ Then for any isometry $\Phi: \Prob_p(X) \rightarrow \Prob_p(X)$ the set of Dirac deltas, $\Delta_1,$ is invariant, i.e. $\Phi(\Delta_1)=\Delta_1.$
\end{thm:A}

This result gives us as a corollary that if two compact m.m.s. $(X,d_X,\mathfrak{m})$ and $(Y,d_Y,\mathfrak{n})$ have isometric
$L^p-$Wasserstein spaces then  $X$ and $Y$ must also be isometric (see Corollary \ref{cor.isometricwass}).

We will need some structure on the metric measure spaces we will we working with, we will assume compactness, non-branching of geodesics and that the reference measure $\mathfrak{m}$ is such that the space has the so called Good transport behaviour $\GTB_p$
for some $p \in (1,\infty).$ Loosely speaking, this last condition requires that optimal transports starting from absolutely continuous measures  are given by a map (see Definition \ref{def.GTBp}). This will also imply that the geodesic in $\Prob_p(X)$ induced by these transports remains  inside the set of absolutely continuous measures until it reaches its endpoint. This condition was first defined in  \cite{GalKellMonSosa} by Galaz-Garc\'ia, Kell, Mondino, and Sosa. Later, it was investigated in more detail by Kell \cite{Kell}.

The class of metric measure spaces that satisfy $\GTB_p$ is quite rich. Examples include Riemannian manifolds, Alexandrov spaces, non-branching $MCP(K,N)$ spaces, and non-branching $RCD^*(K,N)$ spaces.

This question of determining the structure of the group of isometries of the Wasserstein space is not new.  It was first posed in \cite{Klo} by Kloeckner
in the setting of Euclidean spaces, and later for Hadamard spaces in \cite{BerKlo} in collaboration with Bertrand.  In the latter  isometric rigidity is proved. While in the former, more exotic isometries appear, being the
case of the line the most interesting. (see Lemmas $5.2, 5.3$ in \cite{Klo}).
 Some difficulties arise when working in a compact setting; most of the machinery used  previously is no longer available, for example the uniqueness of barycenters is generally no longer true in compact spaces.

In order to obtain that all the isometries of the Wasserstein space come from isometries
of the space we assume an additional hypothesis. Namely we will work on Riemannian manifolds
with strict positive sectional curvature.

\newtheorem*{thm:B}{Theorem \ref{teo:B}}
\begin{thm:B}
Let $M$ be a closed Riemannian manifold with strictly positive sectional curvature. Then
it is isometrically rigid, that is, the isometry groups of $M$ and $\Prob_2 (M)$
coincide.
\end{thm:B}

It is also possible to formulate this same question for different $L^p-$Wasserstein spaces.
In the case of $\mathbb{R}$ and $[0,1]$ this has been done by Geh\'er, Titkos and Virosztek \cite{GehTitVir}. In said
paper they prove that depending on the exponent $p$ the behaviour of the Wasserstein isometries can change,
being the case $p=1$ the most exotic. The same authors have also treated the cases of discrete  and Hilbert spaces 
in \cite{GehTitVir2} and \cite{GehTitVir3} respectively.

In this paper we can answer this question for a certain class of Riemannian manifolds: Compact Rank One Symmetric Spaces (CROSSes). These spaces have nice properties (see Subsection \ref{subsec.CROSS})  that give us enough information on the limitations  that Wasserstein isometries must have.

\newtheorem*{thm:C}{Theorem \ref{teo:C}}
\begin{thm:C}
Let $M$ be a CROSS. Then for any $p \in (1,\infty)$ the isometry groups of $M$ and $\Prob_p (M)$
coincide.
\end{thm:C}

As for other works where compactness is assumed Virosztek \cite{Vir} proved that  in the particular case of the sphere $L^2-$Wasserstein
isometries must send Dirac deltas to Dirac deltas. Theorem \ref{teo:A} only relies on the compactness
of the space as well as structural properties of the optimal transport there (see Definition \ref{def.GTBp}).

The structure of the paper is the following: Section \ref{Section.Preliminaries} is devoted to the presentation of the optimal transport problem and the existence of solutions to it. Some of the geometric properties of the Wasserstein space such as the structure of the optimal transport when one makes certain assumptions on the reference measure $\mathfrak{m}$ are presented as well. Section \ref{Section.Deltas} contains the proof of Theorem \ref{teo:A}. And finally, in Section \ref{Section.rigidity} we prove isometric rigidity in the context of positive curvature for $p=2$ and in the general case $p \in (1,\infty)$ for CROSSes.

%%%%%%%%%%%%%%%%%%%%%%%%%%%%%%%%%%%%%%%%%%%%%%%%%%%%%%%%%%%%%%%%%%%%%%%%%%%%%%%%
% Acknowledgements.
%%%%%%%%%%%%%%%%%%%%%%%%%%%%%%%%%%%%%%%%%%%%%%%%%%%%%%%%%%%%%%%%%%%%%%%%%%%%%%%%
\subsection*{Acknowledgements.}
The author would like to express his thanks to his advisor Prof. Luis Guijarro for valuable comments made during the development
of this paper as well as for his careful reading of earlier versions of this manuscript.
%%%%%%%%%%%%%%%%%%%%%%%%%%%%%%%%%%%%%%%%%%%%%%%%%%%%%%%%%%%%%%%%%%%%%%%%%%%%%%%%
%     Preliminaries.
%%%%%%%%%%%%%%%%%%%%%%%%%%%%%%%%%%%%%%%%%%%%%%%%%%%%%%%%%%%%%%%%%%%%%%%%%%%%%%%%
\section{Preliminaries.}\label{Section.Preliminaries}

In this section we will review the concepts on Optimal Transport used throughout
the paper as well as the notation used. Throughout the following $(M,d)$ will
be a closed Riemannian manifold equipped with its usual distance. The proofs
of the results presented in this section can be found in \cite{AmbGig}

\subsection{Optimal Transport.}\label{Subsection.OT}

Let  $\mu, \nu$ be two probability measures supported on a m.m.s. $(X,d,\mathfrak{m})$
and $p \in (1,\infty).$ Kantorovich's problem consists of minimizing the functional:
\begin{equation*}
\pi \mapsto \int d^p(x,y) d\pi(x,y) \label{Kantorovich.problem} \tag{KP}
\end{equation*}
among all admissible measures $\pi \in \Adm(\mu,\nu).$ The set $\Adm(\mu,\nu)$
consists of measures in $\pi \in \Prob (M\times M)$ that have marginals $\mu$
and $\nu,$ i.e.
$$\pi \left(A\times M\right)=\mu(A), \pi \left(M\times B\right)= \nu(B), \quad \forall A,B \in \mathcal{B}(M). $$

The intuition behind admissible plans is the following: If $\pi \in \Adm(\mu,\nu)$ then for $A\times B \in \mathcal{B}(M\times M)$ the value $\pi(A\times B)$ holds the information of how the mass from $A$ is sent to $B.$

Observe that the functional \ref{Kantorovich.problem} is linear and that $\Adm(\mu,\nu)$ is a convex closed set (in the narrow topology). It will turn out that Kantorovich's problem always has a solution (see for example Theorem $1.5$ in \cite{AmbGig}).
Measures that minimize \ref{Kantorovich.problem} will be called optimal transports (or optimal plans). The set of optimal transports
between two measures $\mu $ and $\nu$ will be denoted by $\Opt(\mu,\nu).$

Given a probability measure $\pi \in \Adm(\mu,\nu)$ it would be useful to determine whether
it is optimal or not. Intuitively a point $(x,y) \in \supp \pi$ represents
the mass that is sent from $x$ to $y.$ So if our plan is optimal then there shouldn't way to rearrange the points in $\supp \pi$ in a way that decreases the value
of the functional. More rigorously we have:

\begin{defi}\label{def.pcyclical}
Let $p \in (1,\infty),$ we say that a set $\Gamma \in X \times X$ is $p-$cyclically monotone if for all $n \in \mathbb{N}$
and $(x_1,y_1), \cdots, (x_n,y_n) \in \Gamma$ implies
$$\sum_{i=1}^{n}d^p(x_i,y_i) \leq \sum_{i=1}^{n}d^p(x_i,y_{\sigma (i)})$$
for all permutations $\sigma$ of $\lbrace 1, \cdots, n\rbrace.$
\end{defi}

It is proved in Theorem $2.13$ \cite{AmbGig} that a plan $\pi$ is optimal if and only if its
support $\supp \pi$ is $p-$cyclically monotone.

When there exists a measurable map $T: X \rightarrow X$ such that $T_{\#}\mu= \nu$
and the plan $(Id,T)_{\#}\mu$ is optimal we will say that the optimal transport is induced by a map.
As a matter of fact the problem of finding a map $T$ that is optimal is the original transport problem
posed by Monge  (see for example Chapter $3$ of \cite{Vill}).

In general it is not possible to find an optimal transport induced by such a map. Brenier (for Euclidean spaces) \cite{Bre} and McCann (for Riemannian spaces) \cite{McCann}  showed that if one takes the starting measure $\mu $ to be absolutely continuous such a map exists. In the next subsection we will describe some other spaces where this also possible.

\subsection{Wasserstein space.}

Let $p\in (1,\infty),$ using the solutions to the Kantorovich problem \ref{Kantorovich.problem} it is
possible to define a metric on $\Prob_p (X).$ The  $L^p-$Wasserstein metric.
Let $\mu ,\nu \in \Prob_p(X)$ then

$$\W_p^p (\mu,\nu) := \min\bigg\lbrace \int d^2(x,y)d\pi \,|\, \pi \in \Adm(\mu,\nu) \bigg\rbrace $$

\begin{Obs}
Usually $\Prob_p(X)$ denotes the space of probability measures with finite $p-$moments. Since our base spaces will always be compact then
it is clear every probability measure has finite $p-$moments for all $p.$ We will keep the subindex in the notation just to stress that
we are considering the $L^p-$Wasserstein metric.
\end{Obs}

Given $n \in \mathbb{N}$ we define the set of totally atomic measures:

$$\Delta_n (X) := \bigg\lbrace \mu \in \Prob_p(X) \,|\, \mu = \sum_{i=1}^{n}a_i\delta_{x_i}, x_i \in X, \sum_{i=1}^{n}a_i =1, a_i>0 \bigg\rbrace. $$

It is a standard result that  $\overline{\bigcup_{n \in \mathbb{N}} \Delta_n (X)}^{\W_p}$ (Theorem $6.18$ in \cite{Vill}). If there is no confusion on which underlying space we are working with
 we will simplify the notation and use instead just $\Delta_n.$

The Wasserstein space will share many geometrical properties with the base space $X$ the first of these is that
the Wasserstein space $\Prob_p(X)$ is compact if and only if $X$ is compact.

A metric space $(X,d)$ is said to be geodesic if for every $x,y \in X$ there
exists a curve $\gamma: [0,1] \rightarrow X $ such that $\gamma_0=x, \gamma_1=y$ and

$$d(\gamma_s,\gamma_t)= |s-t|d(\gamma_0,\gamma_1), \quad s,t \in [0,1]. $$

The set of geodesics of $X$ will be denoted as $\Geo(X).$  It will turn out that this
is sufficient for the existence of geodesics in $\Prob_p(X).$

\begin{teo}[\bf Ambrosio, Gigli \cite{AmbGig} ]
Let $(X,d)$ be a geodesic space. Then the Wasserstein space
 $(\Prob_p(X),\W_p)$ $p\in (1,\infty)$ is geodesic as well.
\end{teo}

\begin{defi}\label{def.nonbranch}
A geodesic space $(X,d)$ will be said to be non-branching if the map
\begin{align*}
\Geo(X) &\rightarrow X \times X \\
\gamma &\mapsto (\gamma_0,\gamma_t)
\end{align*}
is injective for all $t \in (0,1).$
\end{defi}

The property of being a non-branching geodesic space is inherited by the Wasserstein space as the next result
shows:

\begin{teo}[\bf Ambrosio, Gigli \cite{AmbGig}]\label{teo.interiorregularity}
Let $(X,d,\mathfrak{m})$ be a complete and separable m.m.s. Then the space  $(\Prob_p(X),\W_p)$  is
non-branching. Furthermore, given a geodesic $(\mu_t)_{t \in [0,1]}\subset \Prob_p(X)$
then for every $t \in (0,1)$ there exists a unique optimal plan $\Opt(\mu_0,\mu_t)$
and this plan is induced by a map from $\mu_t.$
\end{teo}

\begin{Obs}\label{Obs.measuresinterior}
One of the immediate consequences of the previous Theorem is that measures in $\Delta_{n}(X)$
can only be in the interior of geodesics with endpoints in $\Delta_{1}(X)\cup \cdots \cup \Delta_{n}(X).$
This will be extremely useful for determining how Wasserstein isometries behave.
\end{Obs}

To conclude this section we will describe with more detail the assumptions we will make on the reference measure
$\mathfrak{m}$ and the consequences they will have on the solutions to Kantorovich's  problem. Details can be found
in the paper \cite{Kell} by Kell and the refrences therein.

\begin{defi}\label{def.qualnondeg}
A measure $\mathfrak{m}$ is said to be qualitatively non-degenerate if for all $R>0$
and $x_0\in X$ there is a function $f_{R,x_0}: (0,1)\rightarrow (0,\infty)$ such
that $$\limsup_{t\rightarrow 0}f_{R,x_0}(t)>\frac{1}{2}  $$ and for every
measurable set $A \subset B_{x_0}(R)$ and all $x \in B_{x_0}(R),$ $t \in (0,1)$ we have:
$$\mathfrak{m}(A_{t,x}) \geq f_{R,x_{0}}(t)\mathfrak{m}(A). $$
Where $A_{t,x}:= \{\gamma_t\,|\, \gamma \in \Geo(X), \gamma_0\in A, \gamma_1=x \}. $
\end{defi}

This kind of measures will give us some topological information on the space:

\begin{lem}\label{lem.finitedimension}
Let $(X,d)$ a metric space and $\mathfrak{m}$ a qualitatively non-degenerate probability measure on it.
Then $\mathfrak{m}$ is doubling and $X$ has finite Hausdorff dimension.
\end{lem}

\begin{proof}
The first affirmation is proved in Proposition $5.3$ in \cite{Kell}, as for the second, notice that since
$(X,d)$ is a doubling space then its Assouad dimension (see definition $10.15$ in \cite{Hei}) is finite.
Finally this implies that the Hausdorff dimension of $X$ must be finite.
\end{proof}

As for optimal transports induced by maps we will  recall definition given by Kell in \cite{Kell}:

\begin{defi}\label{def.GTBp}
A m.m.s $X,d,\mathfrak{m}$ is said to have good transport behaviour $\GTB_p$ if for all
$\mu, \nu \in \Prob_p(X)$ such that $\mu \ll \mathfrak{m}$ then any optimal transport between
$\mu$ and $\nu$ is given by a  transport map.
\end{defi}

\begin{example} The following spaces have good transport behaviour:
\begin{enumerate}
  \setlength\itemsep{1em}
  \item For $p=2$ essentially non-branching $MCP(K,N)-$spaces with $K \in \mathbb{R}, N\in [1,\infty).$ Particularly essentially non-branching $CD^*(K,N)-$spaces, essentially non-branching $CD(K,N)$ spaces and $RCD^*(K,N)-$spaces.
  \item $p-$essentially non-branching, qualitatively non-degenerate spaces (Definition \ref{def.qualnondeg}) for all $p \in (1,\infty).$
  \item Any (locallly) doubling measure $\mathfrak{m}$ on $(\mathbb{R}^n,d_E)$ or on a Riemannian manifold.
\end{enumerate}
\end{example}

\begin{defi}\label{def.stronginter}
Let $(X,d,\mathfrak{m})$ a m.m.s where $\mathfrak{m}$ is qualitatively non-degenerate. We will say that it has the $p-$strong interpolation
property $(\sIP_p)$ for some $p \in (1,\infty)$ if: Given any $\mu_0,\mu_1\in \Prob_p(X)$ with $\mu_0\ll \mathfrak{m}$ there is a unique optimal transport, induced by a map,  and the geodesic  $(\mu_t)_{t \in [0,1]}$ satisfies $\mu_t \ll \mathfrak{m}$ for all $t \in [0,1).$
\end{defi}

As an immediate corollary of Theorem $5.8$ in \cite{Kell}, we have:

\begin{teo}
If $(X,d,\mathfrak{m})$ is a non-branching m.m.s. and $\mathfrak{m}$ is qualitatively non-degenerate then $GTB_p$ and $sIP_p$ are equivalent.
\end{teo}

%%%%%%%%%%%%%%%%%%%%%%%%%%%%%%%%%%%%%%%%%%%%%%%%%%%%%%%%%%%%%%%%%%%%%%%%%%%%%%%%
%     Restricting isometries to Dirac deltas (may change name)
%%%%%%%%%%%%%%%%%%%%%%%%%%%%%%%%%%%%%%%%%%%%%%%%%%%%%%%%%%%%%%%%%%%%%%%%%%%%%%%%
\section{Restricting Wasserstein isometries to Dirac deltas.}\label{Section.Deltas}

For the remainder of this paper we will assume that our  spaces satisfy the following:
  $(X,d,\mathfrak{m})$  is a m.m.s. such that it is compact, non-branching, $\mathfrak{m}$ is qualitatively non-degenerate  and with  $\GTB_p$
  for some  $p \in (1,\infty).$

\begin{prop}\label{prop.strictconvexfunct}
Let $(X,d,\mathfrak{m})$ be a m.m.s. with $\GTB_p$ and $\mu \in \Prob_p(X)$ such
that $\mu \ll \mathfrak{m}$ Then the functional:
$$\nu \mapsto \W_p^p(\mu, \nu),  $$
is linearly strictly convex, That is, for any $\nu_0,\nu_1 \in \Prob_p(X)$ and $t \in (0,1)$
we  have $\W_p^p(\mu, (1-t)\nu_0+t\nu_1)< (1-t)\W_p^p(\mu,\nu_0)+t\W_p^p(\mu,\nu_1). $
\end{prop}

\begin{proof}
  Suppose there exist $\eta_0, \eta_1 \in \Prob_p(X)$ and $t^*\in (0,1)$ such that:
  $$\W_p^p(\mu,(1-t^*)\eta_0+t^*\eta_1) = (1-t^*)\W_p^p(\mu,\eta_0)+t^*\W_p^p(\mu,\eta_1). $$
  Consider now the optimal plans $\pi_0 \in \Adm(\mu,\eta_0) $ and $\pi_1 \in \Adm(\mu,\eta_1). $
  It is clear that $(1-t^*)\pi_0+t^*\pi_1$ is an admissible plan between $\mu $ and $(1-t^*)\eta_0+t^*\eta_1.$

  Then:
  \begin{align*}
  \int d^p(x,y) d(1-t^*)\pi_0+t^*\pi_1 &= (1-t^*)\int d^p(x,y)d\pi_0+t^*\int d^p(x,y)d\pi_1 \\
                                      &= (1-t^*)\W_p^p(\mu,\eta_0)+t^*\W_p^p(\mu,\eta_1) \\
                                      &= \W_p^p(\mu,(1-t^*)\eta_0+t^*\eta_1).
  \end{align*}
  And this gives us a contradiction since the plan $(1-t^*)\pi_0+t^*\pi_1$ cannot be induced by a map.
\end{proof}

And this lemma gives us as a Corollary:

\begin{cor}\label{cor.argmaxDirac}
For $\mu \ll \mathfrak{m},$ $\argmax(\Prob_p(X)\ni \nu \mapsto \W_p^p(\mu,\nu)) \subset \Delta_1.$
\end{cor}

\begin{prop}\label{prop.invabscont}
Let $\Phi \in Iso(\Prob_p(X))$ then the set of absolutely continuous measures $\mu$ such that $\Phi(\mu)$ is also absolutely continuous is dense in $\Prob_p(X).$
\end{prop}

\begin{proof}
Let $\mu, \nu \ll \mathfrak{m},$ and consider the geodesic $(\mu_t)_{t\in [0,1]}$ such that $\mu_0=\mu$ and $\mu_1= \Phi^{-1}(\nu).$ Since the m.m.s. has $\GTB_p$ we have that $\mu_t \ll \mathfrak{m}$ for all $t \in [0,1),$ now apply the isometry to the geodesic to obtain a new geodesic $(\Phi(\mu_t))_{t \in [0,1]},$ it is clear then that $\Phi(\mu_1)=\nu \ll \mathfrak{m} $ and so we conclude that $\Phi(\mu_t)\ll \mathfrak{m}$ for all $t \in (0,1).$

Since the measures $\mu ,\nu$ were picked arbitrarily we obtain the thesis.
\end{proof}

\begin{rem}
Actually we have a stronger property: For any $\nu \in \Prob_p(X)$ there exists $\mu_0\ll \mathfrak{m}$ with $\Phi(\mu_0)\ll \mathfrak{m}$
such that if $(\mu_t)_{t \in [0,1]}$ is the unique geodesic joining $\mu_0$ with $\nu$ then  $\Phi(\mu_t) \ll \mathfrak{m}$ for all $t \in [0,1).$
\end{rem}

\begin{Obs}\label{Obs.boundarylinconvex}
For every $x \in X$ and $R>0$ the set $\partial B_{\delta_x}(R)$ is compact and linearly convex. This is just a consecuence of the fact that the optimal transport between a $\delta_x$ and any other measure $\mu$ is given by $\delta_x\otimes \mu.$
\end{Obs}

\begin{lem}\label{lem.fixinghull}
Let $x,y \in X$ be two points such that $\Phi(\delta_x),\Phi(\delta_y)\in \Delta_1,$ where $\Phi \in \Iso(\Prob_p(X)).$ Then there exists a
geodesic $\gamma \in \Geo(X)$ such that $\gamma_0=x,\gamma_1=y$ and  $\Phi(\gamma_t)\in \Delta_1$ for all $t \in [0,1].$
\end{lem}

\begin{proof}
Let $x,y \in X$ be such that $\Phi(\delta_x),\Phi(\delta_y)\in \Delta_1.$ Take $\mu \ll \mathfrak{m}$ such that $\Phi(\mu)\ll \mathfrak{m}$ which exists by Proposition \ref{prop.invabscont}.  Now consider:
$$\Mid(\delta_x,\delta_y) := e_{1/2\#}\{ (\eta_t)_{t \in [0,1]} \in \Geo(\Prob_p(X))\,|\, \eta_0 = \delta_x, \eta_1=\delta_y  \}. $$
and notice that this set has the following properties:
\begin{itemize}
\setlength\itemsep{1em}
\item $\Mid(\delta_x,\delta_y) = \partial B_{\delta_x}(\frac{d(x,y)}{2})\cap \partial B_{\delta_y}(\frac{d(x,y)}{2}),$ so by Observation \ref{Obs.boundarylinconvex} is a linearly convex and compact set.
\item $\Phi(\Mid(\delta_x,\delta_y))= \Mid(\Phi(\delta_x),\Phi(\delta_y)).$
\end{itemize}

Since $\Phi$ is an isometry then it is clear that:
$$\max\{\W_p^p(\mu,\nu)\,|\, \nu \in \Mid(\delta_x,\delta_y) \}= \max\{\W_p^p(\Phi(\mu),\nu)\,|\, \nu \in \Mid(\Phi(\delta_x),\Phi(\delta_y)) \}. $$
And by Proposition \ref{prop.strictconvexfunct} and the fact that both $\mu, \Phi(\mu)\ll \mathfrak{m}$
\begin{align*}
\argmax \big(\Mid(\delta_x,\delta_y)\ni \nu \mapsto \W_p^p(\mu,\nu) \big) &\subset \Delta_1\\
\argmax \big(\Mid(\Phi(\delta_x),\Phi(\delta_y))\ni \nu \mapsto \W_p^p(\Phi(\mu),\nu) \big) &\subset \Delta_1
\end{align*}
Hence there must exists some point $z\in X$ such that $\delta_z \in \argmax(\Mid(\delta_x,\delta_y)\ni \nu \mapsto \W_p^p(\mu,\nu))$ and
$\Phi(\delta_z)\in \Delta_1.$
\end{proof}

The idea for proving Theorem \ref{teo:A} boils down to proving that there exists some set $S\subset X $ with non-empty interior such that
for all $x \in S,$ $\Phi(\delta_x)\in \Delta_1.$ Using then Observation \ref{Obs.measuresinterior} gives us the result.  As for how  we build the set $S$ we recall that from Lemma \ref{lem.finitedimension} the Hausdorff dimension of $X$ is finite. So then for sufficiently enough points $x_1,\cdots, x_n$ such that $\Phi(\delta_{x_i})\in \Delta_1$   the geodesic convex hull of $\{x_1,\cdots x_n\} $ will have non-empty interior.

Given a set $E \subset X$ we define the antipodal set of $E$ as:
$$A(E):= \argmax(X\ni x \mapsto d(x,E) ). $$

\begin{thmp}\label{teo:A}
Let $(X,d,\mathfrak{m})$ be a compact metric measure space with $\GTB_p$ for some $p \in (1,\infty).$ Then for any isometry $\Phi: \Prob_p(X) \rightarrow \Prob_p(X)$ the set of Dirac deltas, $\Delta_1,$ is invariant, i.e. $\Phi(\Delta_1)=\Delta_1.$
\end{thmp}

\begin{proof}
Take $\mu \ll \mathfrak{m}$ such that $\Phi(\mu)\ll \mathfrak{m}$ (such measure exists by Proposition \ref{prop.invabscont}). Then by Corollary \ref{cor.argmaxDirac}:
\begin{equation}\label{eq.funct1}
\argmax(\Prob_p(X)\ni \nu \mapsto \W_p^p(\mu,\nu)) \subset \Delta_1
\end{equation}

\begin{equation}\label{eq.funct2}
\argmax(\Prob_p(X)\ni \nu \mapsto \W_p^p(\Phi(\mu),\nu)) \subset \Delta_1
\end{equation}
Since $\W_p^p(\mu,\nu)=\W_p^p(\Phi(\mu),\Phi(\nu))$ the isometry $\Phi$ sends the set \ref{eq.funct1} to the set \ref{eq.funct2}. That is, there exists some $x_1\in X$ such that $\Phi(\delta_{x_1})\in \Delta_1.$

Suppose now that we have found $x_1,\cdots ,x_n \in X$ such that $\Phi(\delta_{x_i})\in \Delta_1.$ By Lemma \ref{lem.fixinghull} the geodesic convex hull
$S(x_1,\cdots, x_n )\subset X$ consists of points with the property that for all $y \in S(x_1,\cdots, x_n),$  $\Phi(\delta_y) \in \Delta_1.$

Now consider the totally atomic measure $\sum_{i=1}^n\frac{1}{n}\delta_{x_i},$  by the density stated in Proposition \ref{prop.invabscont}
we can find some measure $\mu_{n+1}$ such that:
\begin{itemize}
\setlength\itemsep{1em}
\item $\mu_{n+1},\Phi(\mu_{n+1})\ll \mathfrak{m}.$

\item $\supp \mu_{n+1} \subset \sqcup_{i=1}^{n}B_{x_i}(\epsilon_i),$  for some $\epsilon_i >0,$ $i \in \{1,\cdots, n\}.$

\item There exists some $y \in A(S(x_1,\cdots, x_n))$ such that $\W_p^p(\mu_{n+1},\delta_y)\geq \W_p^p(\mu_{n+1},\delta_z)$ for all $z \in S(x_1,\cdots, x_n).$
\end{itemize}
So we just look at the arguments of the maxima for the linearly strictly convex functionals $\W_p^p(\mu_{n+1},\cdot), \W_p^p(\Phi(\mu_{n+1}),\cdot)$ and obtain a
point $x_{n+1}\in X- S(x_1,\cdots, x_n)$ such that $\Phi(x_{n+1})\in \Delta_1.$

Since by Lemma \ref{lem.finitedimension} the Hausdorff dimension of $X$ is finite we have that there must exist  $m\in \mathbb{N}$ such that
$S(x_1,\cdots,x_m)$ has nonempty interior.

So for  points $z \in X$ and $x$ in the interior of $S(x_1,\cdots,x_m)$ a geodesic $\gamma$ with endpoints $y$ and $x$ must satisfy:
There exists $t \in (0,1)$ such that $\gamma_t \in S(x_1,\cdots,x_m).$ So by applying  Observation \ref{Obs.measuresinterior} then $\Phi(\delta_z)\in \Delta_1.$
\end{proof}

And with this we also have as a Corollary:

\begin{cor}\label{cor.isometricwass}
  Let $(X,d_X,\mathfrak{m}),(Y,d_Y,\mathfrak{n})$ be two compact m.m.s such that they have
   $\GTB_p$ for some $p \in (1,\infty).$ Suppose there exists some isometry $ \Phi: \Prob_p(X)\rightarrow \Prob_p(Y).$   Then $X$ and
   $Y$ are isometric.
\end{cor}

\begin{proof}
Let $\Phi: \Prob_p(X)\rightarrow \Prob_p(Y)$ be the isometry. Then from Theorem \ref{teo:A} we have that for all $x \in M,$ $\Phi(\delta_x) \in \Delta_1(Y).$ So we define  $F: X\rightarrow Y $ as the function such that $F_{\#}(\delta_{x}):= \Phi(\delta_{x}).$ It is clear that this defines an isometry of the base metric spaces.
\end{proof}

%%%%%%%%%%%%%%%%%%%%%%%%%%%%%%%%%%%%%%%%%%%%%%%%%%%%%%%%%%%%%%%%%%%%%%%%%%%%%%%%
%     Isometric Rigidity.
%%%%%%%%%%%%%%%%%%%%%%%%%%%%%%%%%%%%%%%%%%%%%%%%%%%%%%%%%%%%%%%%%%%%%%%%%%%%%%%%
\section{Isometric rigidity.}\label{Section.rigidity}

This last section deals with answering affirmatively the question of isometric rigidity, at least in some cases.

\subsection{Positively curved manifolds and $p=2$}

In this subsection we will prove that if we additionally ask that the manifold has
positive sectional curvature then it is possible to obtain isometric rigidity.
We will restrict only to the case where $p=2,$ the reason being that the flatness condition
(see Definition \ref{def.flatspace}) is of no use to us for general $p.$ In the next subsection however, we will prove isometric rigidity
in the case that we have more structure on the manifold.

\begin{thmp}\label{teo:B}
Let $M$ be a closed Riemannian manifold of strictly positive sectional curvature. Then,
the Wasserstein space $(\Prob_2(M),\W_2)$ is isometrically rigid,
i.e. $\Iso(M)= \Iso\Prob_2 (M).$
\end{thmp}

Given an isometry $\Phi: \Prob_2 (M)\rightarrow \Prob_2(M)$ we have obtained from
Theorem \ref{teo:A} that it restricts to an isometry of the Riemannian manifold.
Therefore from here on out we will only consider isometries $\Phi$ of $\Prob_2(M)$ such
that  $\Phi (\delta_x) = \delta_x$ for all $x \in M.$

Our aim will be to prove that for all $n \in \mathbb{N},$  $\mu \in \Delta_n$
$\Phi(\mu)=\mu.$ Since the union of these sets over all $n$ is dense then we will
be done.

Consider $\gamma: [0,1]\rightarrow M$ a geodesic. Then $\gamma_{\#}:(\Prob_2 ([0,1]),\W_2)\rightarrow (\Prob_2(M),\W_2)$
is an isometric embedding of $(\Prob_2 ([0,1]),\W_2)$ into $(\Prob_2(M),\W_2).$ We will deonte the set of probability measures
supported in the geodesic $\gamma$ as  $\Prob_2(\gamma).$

Our first step will be to prove that $\Prob_2(\gamma)$ is not only invariant under isometries $\Phi$ that fix Dirac deltas but
that $\Phi(\mu)=\mu$ for all $\mu \in \Prob_2(\gamma).$

%aqui igual y habria que poner una discusion sobre que la curvatura de este espacio es plana y que en general los unicos espacios planos embebidos en una variedad son las geodesicas.

In general, regardless on any curvature assumption on the manifold $M$ the only possible totally geodesic embedded submanifolds that one can expect to find are precisely the minimizing geodesics. These geodesics are actually flat spaces in the sense that their curvature is identically 0. The next definition gives an alternative formulation of a metric space being flat just in terms of the distance.

\begin{defi}\label{def.flatspace}
Let $(X,d)$ be a geodesic space, we will say that it is flat if given any three points $x,y,z \in X$ and every $\gamma: [0,1]\rightarrow X$
geodesic such that $\gamma_0=y, \gamma_1=z$ we have:
$$d^2(x,\gamma_t) = (1-t)d^2(x,\gamma_0)+td^2(x,\gamma_1)-(1-t)td^2(y,z), \quad \forall t \in [0,1].  $$
\end{defi}

Examples of flat spaces include Hilbert spaces, closed intervals $[a,b]$ equipped with an interior product,  and Wasserstein spaces $(\Prob_2([a,b]),\W_2).$ One important observation to make though is that even if the base space $X$ is flat then $\Prob_2(X)$ in general is only non-negatively curved. (see Example $3.2.1$  in \cite{AmbGig}).

\begin{Obs}\label{Obs.delta2dense}
  Before moving on with the next results we will make some observations regarding the structure of the Wasserstein space of a closed interval say , $[0,1]$ equipped with its usual Euclidean metric. In \cite{Klo} it is noted (Proposition $3.4$ ) that the set $\Delta_2(\mathbb{R})$  plays a special role as the convex hull of it is dense.  This is used in order to define the exotic isometries (see Lemma $5.3$  in \cite{Klo}). It is clear that the convex hull of $\Delta_2([0,1])$ will also be dense in $\Prob_2([0,1]).$
\end{Obs}

We will use this as well in the proofs of the next Propositions. The next result appears in Section $2.2$ of \cite{GehTitVir}, we include a proof since our arguments are different.

\begin{prop}\label{Prop.intervalrigid}
Let $([0,1],d_E)$ be the interval equipped with its usual Euclidean metric. Then  the Wasserstein space
$(\Prob_2([0,1]),\W_2)$ is isometrically rigid.
\end{prop}
\begin{proof}
First, let us notice the following: For the functional $\mu \mapsto \W_2^2(\delta_{1/2},\mu), \mu \in  \Prob_2([0,1]) $
$$\argmax \left(\mu \mapsto \W_2^2(\delta_{1/2},\mu  \right) = \lbrace (1-\lambda)\delta_0+\lambda\delta_1\,|\, \lambda \in [0,1]\rbrace.  $$
We can mimic the argument done in Theorem \ref{teo:A} to obtain that all Wasserstein isometries send Dirac deltas to Dirac deltas. Hence it suffices to look at isometries $\Phi: \Prob_2([0,1])\rightarrow \Prob_2([0,1])$  such that $\Phi|_{\delta_1}\equiv id.$

It is immediate that $\Phi \left(\lbrace (1-\lambda)\delta_0+\lambda\delta_1\,|\, \lambda \in [0,1]\rbrace\right)= \lbrace (1-\lambda)\delta_0+\lambda\delta_1\,|\, \lambda \in [0,1]\rbrace $ since $\Phi(\delta_{1/2})=\delta_{1/2}.$

Now, for $(1-\lambda)\delta_0+\lambda\delta_1$ suppose $\Phi((1-\lambda)\delta_0+\lambda\delta_1)= (1-t)\delta_0+t\delta_1$ for some $t \in (0,1).$ Then
 $$t= \W_2((1-t)\delta_0+t\delta_1,\delta_0 )= \W_2((1-\lambda)\delta_0+\lambda\delta_1, \delta_0)=\lambda.  $$
So $t=\lambda.$ That is $\Phi$ restricted to $\lbrace (1-\lambda)\delta_0+\lambda\delta_1\,|\, \lambda \in [0,1]\rbrace$ is the identity.

Consider now $(1-\lambda)\delta_a+\lambda\delta_b,$ WLOG $a<b,$  then it is an interior point of the unique geodesic joining $\delta_{a/(1+a-b)} $ with $(1-\lambda)\delta_0+\lambda\delta_1.$

Since $\Phi$ fixes both $\delta_{a/(1+a-b)} $ and $(1-\lambda)\delta_0+\lambda\delta_1$ it must fix the whole geodesic including $(1-\lambda)\delta_a+\lambda\delta_b$ (See Observation \ref{Obs.measuresinterior}),  hence $\Phi$ fixes
$\Delta_2([0,1]).$ Using Observation \ref{Obs.delta2dense} we conclude then that $\Phi \equiv id.$
\end{proof}

\begin{prop}\label{Prop.invariantgeodesics}
Let $x \in M$ and consider $\gamma$ a geodesic starting at $x$ and such that it cannot be extended past $\gamma_1$ while remaining minimizing. Then for any isometry $\Phi$ such that it fixes all Dirac deltas we have that $\Phi (\Prob_2(\gamma))=\Prob_2(\gamma).$
\end{prop}
\begin{proof}
Let $\Phi: \Prob_2 (M)\rightarrow \Prob_2 (M)$ be an isometry such that $\Phi|_{\Delta_1}\equiv id.$ In order to get the result, from the Observation \ref{Obs.delta2dense} it will be sufficient to prove that $\Phi(\Delta_2(\gamma))=\Delta_2(\gamma).$

First, we will prove that the set $\lbrace (1-\lambda)\delta_{\gamma_0}+\lambda\delta_{\gamma_1} \,|\, \lambda\in [0,1]  \rbrace $ is invariant.

Consider two points in $\gamma ([0,1])$ and a geodesic $\sigma :[0,1]\rightarrow M$ joining them. Notice that $\sigma$ is completely contained in $\gamma ([0,1]).$

Let $\mu= (1-\lambda)\delta_{\gamma_0}+\lambda\delta_{\gamma_1},$ since $\Prob_2(\gamma)$ is a flat space we have that
$$\W_2^2(\mu,\delta_{\sigma_t})= (1-t)\W_2^2(\mu,\delta_{\sigma_0})+t\W_2^2(\mu,\delta_{\sigma_1})-(1-t)t\W_2^2(\delta_{\sigma_0},\delta_{\sigma_1}),   $$
for all $t \in [0,1].$ And since $\Phi$ fixes every Dirac delta we have an analogous equation for $\Phi(\mu),$ which can be rewritten in the following way since the product measure $\Phi(\mu)\otimes \delta_{\sigma_t} $ is optimal for all $t \in [0,1].$

$$\int d^2(y, \sigma_t)-(1-t)d^2(y,\sigma_0)-td^2(y,\sigma_1)+t(1-t)d^2(\sigma_0,\sigma_1) d\Phi(\mu)(y) =0. $$
As the manifold is of positive sectional curvature then the integrand must the $\Phi(\mu)-$a.e. identically 0. That is,
$$d^2(y, \sigma_t)-(1-t)d^2(y,\sigma_0)-td^2(y,\sigma_1)+t(1-t)d^2(\sigma_0,\sigma_1))=0,$$ for all $t \in [0,1]$ and  $\Phi(\mu)$-a.e.
The positive curvature then forces the support of $\Phi(\mu)$ to be in $\gamma [0,1],$ otherwise we would have Euclidean traingles embedded in $M.$ Therefore we obtain the thesis.
\end{proof}
Clearly given  a geodesic $\gamma\in \Geo(M)$ we have that $\Prob_2(\gamma)$ and $\Prob_2([0,d(\gamma_0,\gamma_1)])$ are isometric so
combining Propositions \ref{Prop.intervalrigid} and \ref{Prop.invariantgeodesics} we obtain the following corollary:

\begin{cor}
For any geodesic $\gamma: [0,1]\rightarrow M$ and $\Phi: \Prob_2(M)\rightarrow \Prob_2(M)$ an isometry that fixes all Dirac deltas. $\Phi$ restricted to  $(\Prob_2(\gamma),\W_2)$ is the identity map.
\end{cor}

And noting that given any $\mu \in \Delta_2(M)$ its atoms are contained in some geodesic we obtain:

\begin{cor}
Let $\Phi: \Prob_2(M)\rightarrow \Prob_2(M)$ be an isometry such that it fixes Dirac deltas. Then $\Phi$ fixes $\Delta_2(M)$ as well.
\end{cor}

Now we will like to see what happens when we look at measures not necessarily supported on a geodesic.
Given a measure $\mu \in \Prob_2(M)$ and a geodesic $\gamma \in \Geo(M)$ we define $\proj_{\gamma}(\mu)$
the projection of $\mu$ onto $\gamma$ as:

\begin{equation}\label{def.projmeasures}
\proj_{\gamma}(\mu) = \argmin \left( \nu \in \Prob_2(\gamma) \mapsto \W_2^2(\mu,\nu)  \right).
\end{equation}

Note that in general  $\mu \mapsto \proj_{\gamma}(\mu)$ is not a function since the set on right hand side of
\ref{def.projmeasures} may contain more than one element.  For example consider in the sphere $\mu = \delta_N$
the north pole and $\gamma $ a geodesic  in the equator. It is clear that every measure in $\Prob_2(\gamma)$ is
equidistant to $\Delta_N.$

Nevertheless it will be very useful to us to work with projections onto geodesics. It is easy to convince
oneself that if $\mu$ is a totally atomic measure in say, $\Delta_n(M)$ the projection onto any geodesic contains
at least one totally atomic measure.

\begin{prop}
Let $\mu \in \Delta_n(M),$ then for every  geodesic $\gamma$
then $proj_{\gamma}(\mu) \cap (\Delta_{1}(M)\cup\cdots \Delta_{n}(M)) \neq \emptyset.$
\end{prop}

\begin{proof}
Let $\mu = \sum_{i=1}^{n}\lambda_i\delta_{x_i}$ and $\gamma \in \Geo(M).$ Take now a measure $\nu \in \Prob_2(\gamma).$

Take $\pi \in \Opt(\mu,\nu),$ so then:
\begin{align*}
\W_2^2(\mu,\nu) &= \int_{M\times M} d^2(x,y) d\pi(x,y) \\
                &= \int_{\lbrace x_1\rbrace \times M} d^2(x,y)d\pi(x,y) + \cdots + \int_{\lbrace x_n\rbrace \times M} d^2(x,y)d\pi(x,y)\\
                &\geq \lambda_1 d^2(x_1,y_1)+\cdots +\lambda_nd^2(x_n,y_n).
\end{align*}
Where the points $y_i$ are such that:
$$ y_i \in \argmin \left(\Delta_1(\gamma)\ni \delta_y \mapsto d^2(x_i,y)  \right),  \quad \forall 1\leq i\leq n. $$
Hence we conclude that
$$\sum_{i=1}^{n}\lambda_i\delta_{y_i}\in \proj_{\gamma}(\mu). $$
\end{proof}

So, let us describe our plan for proving the isometric rigidity. First we will prove that for each $n \in \mathbb{N}$
the set $\Delta_n(M)$ is invariant. Then we will prove that  totally atomic measures supported on a sufficently small ball $B$
are  fixed. A density argument will yield that any measure whose support is contained in $B$ is also fixed.
Finally we will use a non-branching argument to conclude.

We divide each of these steps into several Lemmas to make the argument as clear as posible.

\begin{lem}\label{lem.invariancetub}
Let $\Phi: \Prob_2(M)\rightarrow\Prob_2(M)$ be an isometry that fixes all Dirac deltas, consider $\gamma \in \Geo(M)$ and
$\Prob_2(\gamma)$ the set of  measures supported at $\gamma.$ For $\epsilon \ll 1$ consider:
$$\Tub_{\epsilon}(\Prob_2(\gamma)):= \bigcup_{\mu \in \Prob_2(\gamma)}B_{\mu}(\epsilon)  $$
a tubular neighbourhood around $\Prob_2(\gamma).$ Then for every $n \in \mathbb{N},$ and $\mu \in \Delta_n(M)$ such that
$\supp(\mu)\subset \Tub_{\epsilon}(\Prob_2(\gamma))$ we have that $\Phi(\mu) \in \Delta_n(M).$
 \end{lem}

\begin{proof}
We may assume that all the measures considered are such that they give zero mass to $\gamma[0,1].$
The proof will be done by induction over $n \in \mathbb{N}.$ For $n=1$ the result is clear as it follows from Theorem \ref{teo:A}.
Let $\mu \in \Delta_{n+1}(M)$ be such that $\supp(\mu)\subset \Tub_{\epsilon}(\Prob_2(\gamma)).$ Then we may write
$\mu = \sum_{i=1}^{n+1}\lambda_i\delta_{x_i},$ where all atoms are different and for all $i,$ $\lambda_i\neq 0.$

Now take $\nu \in \proj_{\gamma}(\mu)\cap \Delta_{n+1}(\gamma)$ such that $\nu = \sum_{i=1}^{n+1}\lambda_i\delta_{y_i}.$
Denote by $r_i = d(x_i,y_i).$ Notice that since $\epsilon$ is sufficiently small there exists only one geodesic between
$x_i$ and $y_i,$ and that such geodesic may be extended up to another point $z_i$ such that $d(x_i,z_i)=2r_i.$

Note that this makes $\nu$ the midpoint between $\mu$ and the totally atomic measure $\sum_{i=1}^{n+1}\lambda_i\delta_{z_i}.$
And so we have that since $\Phi(\nu)=\nu$ using Theorem \ref{teo.interiorregularity} and the induction hypothesis this forces
$\Phi(\mu) \in \Delta_{n+1}(M).$
\end{proof}

\begin{lem}\label{lem.invarianceatomic}
Let $\Phi: \Prob_2(M)\rightarrow\Prob_2(M)$ be an isometry that fixes all Dirac deltas, then for every $n \in \mathbb{N}$
$\Phi(\Delta_n(M))=\Delta_n(M).$
\end{lem}

\begin{proof}
Let $\mu \in \Delta_n(M),$ and consider a geodesic $\gamma \in \Geo(M)$ such that for some $\nu \in \proj_{\gamma}(\mu)\cap \Delta_{n}(\gamma).$ Furthermore assume that the transport between these two measures  is given by a map.

Now if we consider the Wasserstein geodesic $(\eta_{t})_{t \in [0,1]}\subset \Delta_n(M)$ between $\Phi(\mu)$ and $\nu=\Phi(\nu)$ there exists some $t_0\in (0,1)$ such
that $\supp (\eta_{t_0})\subset \Tub_{\epsilon}(\Prob_2(\gamma))$ for some sufficiently small $0<\epsilon.$  From the previous Lemma
\ref{lem.invariancetub} we obtain that $\eta_{t_0}\in \Delta_n(M),$ and from Theorem \ref{teo.interiorregularity} we have that
$\Phi(\mu)\in \Delta_n(\gamma).$
\end{proof}

\begin{rem}\label{rem.weights}
Notice that in the proofs of the previous Lemmas \ref{lem.invariancetub}, \ref{lem.invarianceatomic} the transports considered were actually
given by a map. Hence the weights given to each of the  atoms are also preserved.
\end{rem}

\begin{lem}\label{lem.fixedatomic}
Let $\Phi: \Prob_2(M)\rightarrow\Prob_2(M)$ be an isometry that fixes all Dirac deltas, and $n \in \mathbb{N}$
then for every $\mu = \sum_{i=1}^{n}\lambda_i\delta_{x_i} \in \Delta_n(M)$ such that $d(x_i,x_j)< \epsilon$ for all $i \neq j$ and $\epsilon$ sufficently small
we have that $\Phi(\mu)=\mu.$
\end{lem}

\begin{proof}
Given a geodesic $\gamma \in \Geo(M)$ it is clear that $\proj_{\gamma}(\mu)$ consists only of totally atomic measures. It also easy to check
that:
$$\proj_{\gamma}(\mu)= \Phi(\proj_{\gamma}(\mu) ) = \proj_{\gamma}(\Phi(\mu)). $$
Furthermore if we additionally assume that $\gamma_{t_0} = x_1$  for some $t_0 \in [0,1]$ we obtain that for all $\nu \in \proj_{\gamma}(\mu),$  $\nu(\lbrace x_1 \rbrace)\geq \lambda_1.$

From Lemma \ref{lem.invarianceatomic} we know that $\Phi(\mu)$ is also totally atomic, actually with the same number of atoms as $\mu.$
Suppose that there exists some $r>0$ such that $d(x_1,y)>r$  for all $y$ atom of $\Phi(\mu).$ This implies however, that there exists some geodesic $\sigma \in \Geo(M)$ such that $\sigma_0=x_1$ and that for some $\tilde{\nu} \in \proj_{\sigma} (\Phi(\mu))$  $\tilde{\nu}(\lbrace x_1\rbrace) = 0.$  This gives us a contradiction. Hence $x_1$ is one of the atoms of $\Phi(\mu).$

We repeat this argument for every $x_i$ and conclude that $\mu$ and $\Phi(\mu)$ must have the same atoms. From the Remark \ref{rem.weights} we conclude then that $\mu = \Phi(\mu).$
\end{proof}

\begin{lem}\label{lem.nonbranchargument}
$M$ is isometrically rigid.
\end{lem}

\begin{proof}
Take a fixed point $w \in M$ and some $\epsilon \ll  1.$ from the Lemma \ref{lem.fixedatomic} we have that for all $\mu \in \Prob_2(B_w(\epsilon)),$  $\Phi(\mu)=\mu.$

Consider then $\mu \Prob_2(B_w(\epsilon))$ absolutely continuous and $\nu \in \Prob_2(M)$ some arbitrary measure. Additionally
assume that $\supp (\mu) \subset B_{w}(\epsilon/2).$  Let $(\eta_t)_{t \in [0,1]}$ be the unique Wasserstein geodesic such
that $\eta_0 = \mu $ and $\eta_1= \nu.$ Hence it follows that $(\Phi\eta_t)_{t\in [0,1]}$  is also a Wasserstein geodesic but now with endpoints $\mu$ and $\Phi(\nu).$

Then there exists some $t_0 \in (0,1) $  such that  $\supp\Phi(\eta_{t_0}) \subset B_{w}(\epsilon),$ hence $\Phi(\eta_{t_0})=\eta_{t_0}.$ Since the Wasserstein space $\Prob_2(M)$ is non-branching it follows then that $\Phi(\nu)=\nu.$ Therefore $M$ is isometrically rigid.
\end{proof}

%\begin{rem}
%Notice that to obtain the thesis of Theorem \ref{teo:B} it is only needed  that if a Wasserstein geodesic is such that it fixes $\Delta_1$ then it must also fix $\Prob_2(\gamma)$ for all $\gamma \in \Geo(M).$
%\end{rem}

\subsection{Rigidity on CROSSes}\label{subsec.CROSS}

%aqui hace falta poner las propiedades de los CROSSES que nos interesan. Así como también la lista de quiénes son.
In this last subsection we will restrict ourselves to the class of Compact Riemannian Symmetric Spaces (CROSSes). Several properties
of these spaces are discussed in Chapter $3$ of \cite{Bes}.They have been completely  classified and are: Euclidean spheres, projective spaces (with field either real, complex or quaternionic numbers), and the Cayley plane.

In the next Lemma we summarize the properties that will use:

\begin{lem}
Let $M$ be a CROSS. Then:
\begin{itemize}
\item For any point $x \in M$ the isotropy group at $x$ acts transitively on any sphere $\partial B_{x}(R).$

 \item For any $x \in M$ the cut locus, $\Cut(x),$ is either a point or a totally geodesic embedded CROSS of codimension $1.$
\end{itemize}
\end{lem}

An immediate but important consecuence of this Lemma is that for any point $x$ the distance to $\Cut(x)$ is constant and equal to the diameter of $M.$

Using Theorem \ref{teo:A} we will now restrict to isometries that fix every Dirac delta. But before continuing let us describe the motivation of our strategy to prove Theorem \ref{teo:C}. In \cite{Klo} Kloeckner proved  (Proposition $3.4$) that for $n\geq 2$ the geodesic convex hull of $\Delta_1(\mathbb{R})$ is dense in $\Prob_2(\mathbb{R}).$ Thefore it is enough to describe the behaviour of the isometries on $\Delta_1(\mathbb{R}).$ Now, in our setting, fix some point $x \in M$ and consider the Lipschitz function:
$$d(x,\cdot): M \rightarrow [0,\Diam(M)]. $$
The preimage of an element in $\Delta_n([0,\Diam(M)])$ consists of measures whose supports are contained on spheres $\partial B_{x}(r_i)$ $1\leq i \leq n;$ for example, measures in $\Delta_n(M)$ are included here. Therefore a naive (but ultimately useful) approach would be to look at the behaviour of Wasserstein isometries at these measures. More precisely we have:

\begin{prop}\label{prop.parallelfixed}
Let $M$ be a CROSS, $\Phi \in \Iso(\Prob_p(X)),$ $p \in (1,\infty)$ and assume that for all $x \in M$ every probability measure supported on $\Cut(x)$ is fixed by $\Phi.$ Then, for $\mu \in \Prob_p(M)$ such that:
$$\supp (\mu) \subset \bigcup_{i=1}^{n} \partial B_{x}(r_i), \quad 0\leq r_i \leq d(x,\Cut(x)) $$
we have that
$$\supp (\Phi(\mu)) \subset \bigcup_{i=1}^{n} \partial B_{x}(r_i), \quad 0\leq r_i \leq d(x,\Cut(x)). $$
Moreover, the weights are preserved, i.e.  $\mu (\partial B_{x}(r_i)) = \Phi(\mu)(\partial B_{x}(r_i))$ for all $ i \in \{1,\cdots, n\}.$
\end{prop}

Before proving this Proposition we will need a couple  simple observations and an auxiliary lemma:

\begin{Obs}\label{obs.midpointinvariant}
  The following are simple properties of geodesics in the Wasserstein space $\Prob_p(M).$
\begin{itemize}
  \setlength\itemsep{1em}
\item If $\mu_0, \mu_1 \in \Prob_p(M)$   are fixed by $\Phi\in \Iso(\Prob_p(M))$ then the set of geodesics with endpoints $\mu_0$ and $\mu_1$ is invariant under $\Phi.$

\item If $\mu_0 =\delta_x,$ and  $\mu_1$ such that  $\supp \mu_1 \subset \Cut(x).$  Then for any $(\eta)_{t \in [0,1]} \in \Geo(\Prob_p(X))$ joining them we have that $\supp \eta_t \subset \partial B_{x}(td(x,\Cut(x)))$ for all $t \in [0,1].$

\item If $\mu_0= \delta_x$ and $\nu$ is supported on some $\partial B_{x}(td(x,\Cut(x))), 0<t<1$  then there exists a measure $\mu_1$ supported on $\Cut(x)$ such that $\nu$ is in the interior of some geodesic joining $\mu_0 $ and $\mu_1.$ This is clear, just notice that
every point in the support of $\nu$ is of the form $\gamma_t$ for some geodesic starting at $x.$ So just extending these geodesics yields
the measure $\mu_1.$
\end{itemize}
\end{Obs}

\begin{lem}\label{lem.interiorparallel}
  Let $M$ be a CROSS, $\Phi \in \Iso(\Prob_p(M)),$ $p \in (1,\infty)$ and assume that for all $x \in M$ every probability measure supported on $\Cut(x)$ is fixed by $\Phi.$ Then, for $\mu \in \Prob_p(M)$ such that:
  $$\supp (\mu) \subset \bigcup_{i=1}^{n} \partial B_{x}(r_i), \quad 0\leq r_i \leq d(x,\Cut(x)). $$
   Then
  \begin{itemize}
    \setlength\itemsep{1em}
    \item For $n=2$ there exist $\nu_0$ supported on $\{x\}\cup \Cut(x)$
   and $\nu_1$ supported on some $\partial B_x(r) $ such that the  geodesic between them passes through $\mu.$
   \item If $n\geq 3$ there exist $\nu_0,\nu_1 \in \Prob_p(X)$  supported on $n-1$ spheres $\partial B_{x}(r_i)$ and
    such that the  geodesic between them passes through $\mu.$
  \end{itemize}
\end{lem}

\begin{proof}
It will be sufficient to consider totally atomic measures $\mu$ such that they satisfy the following: If $y \in \supp \mu$ and $\gamma \in \Geo(X)$ is such that $\gamma_0=x $ and $\gamma_t=y$ for some $t \in [0,1]$ then $\gamma[0,1]\cap \supp \mu = \{y\}.$ Let $D$ denote the
diameter of $M.$

\begin{case}[\bf $n=2$]
Consider $\mu = (1-\lambda)\mu_{1}+\lambda\mu_{2}$ where $\supp \mu_{i}\subset \partial B_{x}(r_{i})$ and $r_{1}<r_{2}<D.$ As $\mu_{1}$ is supported on a single sphere then from Observation \ref{obs.midpointinvariant} we can find some measure $\tilde{\mu},$ with
 $\supp \tilde{\mu}\subset \Cut(x)$ such that $\mu_{2}$ is in the interior of some geodesic joining $\delta_{x}$ with $\tilde{\mu}.$ We will take then $\nu_0 = (1-\lambda)\delta_x+ \lambda \tilde{\mu}.$
As for $\nu_1$ we do the following: Every point in $\supp \mu_i$ is of the form $\gamma_{t_i},$ where $t_i = r_i/D, $ $\gamma \in \Geo(M)$ starting at $x$ and $i=\{1,2\}.$ So we can just send the mass of these points to $\gamma_{t_1/(1+t_1-t_2)}.$

Therefore the measure $\nu_1$ will be a totally atomic measure whose atoms are of the form $\gamma_{t_1/(1+t_1-t_2)}$ with mass either $\mu_1(\gamma_{t_1})>0$ or $\mu_2(\gamma_{t_2})>0.$ It is easy to see that the  geodesic joining $\nu_0$ with $\nu_1$ passes through $\mu.$
\end{case}

\begin{case}[\bf $n\geq 3$]
Let now $\mu = a_{1}\mu_1+\cdots +a_{n}\mu_n,$ $\supp\mu_i \subset \partial B_{x}(r_i),$  $a_i>0,$ and $0<r_1< \cdots <r_n<D.$
We will do a similar construction as in the previous case. Take
\begin{align*}
\nu_0 &= a_1\mu_1+a_2\eta+ a_3\mu_3+\cdots+a_n\mu_n,\\
\nu_1 &= a_1\mu_1+a_2\tilde{\eta}+a_3\mu_3+\cdots +a_n\mu_n.
\end{align*}
Where $\supp\eta \subset \partial B_{x}(r_1),$ $\supp \tilde{\eta}\subset \partial B_{x}(r_3).$ These measures $\eta, \tilde{\eta}$ are obtained by sending the mass of $\mu_2$ along the appropiate geodesics to $\partial B_{x}(r_1)$ and $\partial B_{x}(r_3)$ respectively.
It is clear that $\mu$ is in the interior of the  geodesic between $\nu_0$ and $\nu_1$
\end{case}
\end{proof}

With this previous result we can now prove the Proposition.

\begin{proof}[Proof of Proposition \ref{prop.parallelfixed}]
We will proceed by induction on the dimension of $M.$ Take $x \in M,$ since $\Cut(x)$ is either a CROSS or a point then we can assume
that the  measures supported there are fixed by $\Phi.$ Denote by $D$ the diameter of $M.$

Take $\mu \in \Prob_p(X)$ and let  $n$ bethe number of spheres $\partial B_x(r_i)$ on which the support of $\mu$
is contained. The case $n=1$ follows clearly from Observation \ref{obs.midpointinvariant}.
\begin{case}[\bf $n=2$]
Consider a measure of the form $(1-\lambda)\delta_x+\lambda \mu $ where $\mu$ is supported on $\Cut(x)$ and absolutely continuous
with respect to the volume measure of $\Cut(x).$

Observe that the geodesic from $\mu$ to $(1-\lambda)\delta_x+\lambda \mu$ cannot be extended further since the length of the
geodesics involved in the optimal transport is already maximal. This implies that there must exists some set $\Gamma \subset \supp \pi,$
$\pi \in Opt(\mu,(1-\lambda)\delta_x+\lambda \mu)$ such that $\pi(\Gamma)>0$ and for all $(y,z)\in \Gamma, d(y,z)=D.$ Also, from the
$p-$cyclical monotone condition (see Definition \ref{def.pcyclical}) for all $(y_1,z_1),(y_2,z_2)$ $d(y_1,z_2)=d(y_2,z_1)=D.$

Now, since $\mu(e_0\Gamma)>0$ and $\mu$ is absolutely continuous with respect to the volume measure on $\Cut(x)$ it follows that
$e_1\Gamma = \{x\}.$  So we deduce that $\Phi((1-\lambda)\delta_x+\lambda\mu)= (1-\alpha)\delta_x+\alpha\nu,$  $\alpha \in (0,1)$
and $\nu \in \Prob_p(M).$

Consider now $\eta \in \Mid(\delta_x,\mu),$ so
\begin{align*}
\W_p^p(1-\alpha)\delta_x+\alpha\nu,\eta) &\leq (1-\alpha)\W_p^p(\delta_x,\eta)+\alpha\W_p^p(\nu,\eta)\\
                                         &= (1-\alpha)(\frac{D}{2})^p+\alpha\W_p^p(\nu,\eta).
\end{align*}
If there exists a set $\bar{\Gamma}\subset \bar{\pi},$ $\pi \in \Opt(\eta,\nu)$ such that $\bar{\pi}(\bar{\Gamma})>0$
and for all $(y,z)\in \bar{\Gamma}$ $d(y,z)< (D/2)^p$ then  $\W_p^p(\nu,\eta) < (D/2)^p.$ But this contradicts the fact
that $\W_p^p((1-\lambda)\delta_x+\lambda\mu,\Phi^{-1}(\eta)) = (D/2)^p. $

As before, the $p-$cyclical monotonicity of the support of $\bar{\pi}$ guarantees that for every $z \in \supp \nu$ $d(x,z)=D,$ i.e.
$\supp \nu \subset \Cut(x).$

Finally, $\W_p^p((1-\lambda)\delta_x+\lambda\mu, \mu)= (1-\lambda)D^p, \W_p^p((1-\lambda)\delta_x+\lambda\mu, \delta_x)=\lambda D^p $
forces $\alpha =\lambda$ and $\nu=\mu.$

The case proved just now is sufficient. First, from the density of the absolutely continuous measures supported on $\Cut(x)$ we can extend
it for all $\mu \in \Prob_p(M),$ with $\supp \mu \subset \Cut(x).$
Next, given a measure $\nu$ supported on $\partial B_x (r_1)\cup \partial B_x (r_2),$ $0< r_1<r_2<D $ there exists by Lemma \ref{lem.interiorparallel} measures $\mu_0, \mu_1$ supported on $\{x\}\cup \Cut(x), \partial B_x(r_1/(1+r_1-r_2))$ respectively such that
$\nu$ is in the interior of the geodesic joining $\mu_0$ with $\mu_1.$ As previously proved both $\Phi(\mu_0), \Phi(\mu_1)$ are supported on the same spheres and this forces $\Phi(\nu)$ to do so as well.
\end{case}

\begin{case}[\bf $n\geq 3$]
By an induction argument on $n$ we obtain the thesis. Lemma \ref{lem.interiorparallel} tell us that measures $\mu$ supported on $n$ spheres
lie in the interior of a geodesic joining two measures supported on $n-1$ spheres.  So by the induction hypothesis the endpoints when we apply the isometry $\Phi$ are supported on the same spheres. Hence this also happens to the support of $\Phi(\mu).$
\end{case}
\end{proof}

And so finally we prove the main Theorem of this subsection.

\begin{thmp}\label{teo:C}
Let $M$ be a CROSS. Then for any $p \in (1,\infty)$ the isometry groups of $M$ and $\Prob_p (M)$
coincide.
\end{thmp}

\begin{proof}
Let us do induction on the dimension of the space $M.$ Take $\mu \in \Delta_n,$  such that $\mu = \sum_{i=1}^n a_i\delta_{x_i}$ where
$a_i > 0, \sum_{i=1}^n a_i=1.$
Fix $x_1$ and notice that $\Cut(x_1)$ is either a point or a totally geodesic embedded CROSS of dimension one less than the dimension of $X.$ By the induction hypothesis we have that any probability measure supported on $\Cut(x_1)$ is fixed by $\Phi.$

Since $\mu$ satisfies  the hypothesis of Proposition \ref{prop.parallelfixed} for $r_i = d(x_1,x_i), i \in \{1,\cdots, n \}$ we obtain then:
$$\Phi(\mu) = a_1\delta_{x_1}+\sum_{i=2}^{n}a_i\mu_i, \text{ where } \supp(\mu_i)\subset \partial B_{x_1}(r_i). $$
We repeat this argument for the remaining $x_i$ and obtain that $\mu=\Phi(\mu).$ Since the clousure of totally atomic measures is dense in
$\Prob_p(M)$ we conclude that $\Phi$ must be the identity.
\end{proof}

%%%%%%%%%%%%%%%%%%%%%%%%%%%%%%%%%%%%%%%%%%%%%%%%%%%%%%%%%%%%%%%%%%%%%%%%%%%%%%%%
%     Bibliography
%%%%%%%%%%%%%%%%%%%%%%%%%%%%%%%%%%%%%%%%%%%%%%%%%%%%%%%%%%%%%%%%%%%%%%%%%%%%%%%%

\end{document}